\documentclass[11pt,a4paper]{article}
\usepackage{amsfonts,amsgen,amsmath,amstext,amsbsy,amsopn,amsthm,amsfonts,amssymb,amscd}
\usepackage{epsf,epsfig}
\usepackage{float}
\usepackage{ebezier,eepic}
\usepackage{color}
\usepackage{tikz}
\usepackage{multirow}
\usepackage{graphicx}
\usepackage{subfigure}
\usepackage{mathrsfs}
\usepackage{graphics}
\usepackage{graphicx}
\usepackage{color}
\usepackage{ifpdf}
\usepackage{fancyhdr}
\usepackage{epstopdf}
\usepackage{epsfig}
\usepackage{indentfirst}
\usepackage{cases}
\usepackage{mathtools}

\newtheorem{thm}{Theorem}[section]

\newtheorem{lem}[thm]{Lemma}
\newtheorem{false statement}{False statement}

\theoremstyle{definition}

\makeatletter \@addtoreset{equation}{section}

\def\hg{\mathcal{G}}

\begin{document}
\title{Extremal triangle-free graphs with chromatic number at least four}
\author{Sijie Ren\footnote{Department of Mathematics, Taiyuan University of Technology, Taiyuan 030024, P. R. China. E-mail:rensijie1@126.com. }\quad\quad
Jian Wang\footnote{Department of Mathematics, Taiyuan University of Technology, Taiyuan 030024, P. R. China. E-mail:wangjian01@tyut.edu.cn. Research supported by NSFC No.12471316.} \quad\quad
Shipeng Wang\footnote{Department of Mathematics, Jiangsu University, Zhenjiang, Jiangsu 212013, P. R. China. E-mail:spwang22@ujs.edu.cn. Research supported by NSFC No.12001242.} \quad\quad
Weihua Yang\footnote{Department of Mathematics, Taiyuan University of Technology, Taiyuan 030024, P. R. China. E-mail:yangweihua@tyut.edu.cn. Research supported by NSFC No.12371356. }\\
 }
\date{}
\maketitle

\begin{abstract}
Let $G$ be an $n$-vertex triangle-free graph.  The celebrated Mantel's theorem showed that  $e(G)\leq \lfloor\frac{n^2}{4}\rfloor$. In 1962, Erd\H{o}s  (together with Gallai), and independently Andr\'{a}sfai, proved that if $G$ is non-bipartite then $e(G)\leq \lfloor\frac{(n-1)^2}{4}\rfloor+1$.  In this paper, we extend this result and show that if $G$ has chromatic number at least four and $n\geq 90$, then  $e(G)\leq \lfloor\frac{(n-3)^2}{4}\rfloor+5$.
The blow-ups of Gr\"{o}tzsch graph show that this is best possible.
\end{abstract}

\noindent{\bf Keywords:} stability; odd cycles; triangle-free graphs; Gr\"{o}tzsch graph.

\section{Introduction}

Let $G=(V,E)$ be a simple undirected graph with vertex set $V(G)$ and edge set $E(G)$. We use $e(G)$ to denote the number of edges of $G$.  A graph $G$ is called {\it $F$-free} if it does not contain a copy of $F$ as a subgraph. The {\it Tur\'{a}n number} ${\rm ex}(n,F)$ is defined as  the maximum number of edges in an $F$-free graph on $n$ vertices. An $F$-free graph on $n$ vertices with ${\rm ex}(n,F)$ edges is called an extremal graph for $F$. In 1907, Mantel proved a celebrated result.

\begin{thm}[\cite{Mantel}]\label{thm-1.1}
	$${\rm ex}(n,K_3)= \left\lfloor\frac{n^2}{4}\right\rfloor.$$
\end{thm}

In 1962, Mantel's theorem was refined by the following result, which was proved by Erd\H{o}s  (together with Gallai), and independently by Andr\'{a}sfai.

\begin{thm}[\cite{erdos2}]\label{erdos2}
	Let $G$ be a non-bipartite triangle-free graph on $n$ vertices. Then
	$$e(G)\leq \left\lfloor\frac{(n-1)^2}{4}\right\rfloor+1.$$
\end{thm}

Let $T_r(n)$ denote the complete $r$-partite graph on $n$ vertices whose part sizes differ by at most $1$ and let $t_r(n)=e(T_r(n))$. Let $H_0$ be a graph obtained from $T_2(n-1)$ by replacing an edge $xy$ with a path $xzy$, where $z$ is a new vertex. It is easy to see that $H_0$ is a triangle-free graph with
 $\left\lfloor\frac{(n-1)^2}{4}\right\rfloor+1$ edges, which shows that  Theorem~\ref{erdos2} is best possible.

In 1941, Tur\'{a}n \cite{turan1941} showed that if $G$ is a $K_{r+1}$-free graph on $n$ vertices then $e(G)\leq t_r(n)$, with equality if and only if $G=T_r(n)$. In 1981, Brouwer refined Tur\'{a}n's theorem in the following form, which is an extension of Theorem~\ref{erdos2}.

\begin{thm}[\cite{brouwer}]\label{brouwer}
Let G be a non-$r$-partite $K_{r+1}$-free graph on $n$ vertices.  Then
	$$e(G)\leq  t_r(n)-\left\lfloor \frac{n}{r}\right\rfloor +1.$$
\end{thm}
 This phenomenon was also studied in \cite{Amin}, \cite{Hanson}, \cite{kang} and \cite{Tyomkyn}. A proper {\it$k$-coloring} of  $G$ is a mapping  $\sigma : V(G)\to \{1, \cdots, k\}$ such that $\sigma(x)\neq\sigma(y)$ for every $xy\in E(G)$. The {\it chromatic
	number  $\chi(G)$} of $G$ is the minimum number $k$ such that $G$ has a proper $k$-coloring. We say that $G$ is a $k$-chromatic graph if $\chi(G) = k$.

Let us introduce the Gr\"{o}tzsch graph $\Gamma$ (as shown in Figure 1), which has 11 vertices and 20 edges. It was showed in \cite{chvatal} that $\Gamma$ has the fewest vertices among all triangle-free $4$-chromatic graphs.

\begin{figure}[H]
\centering
\ifpdf
  \setlength{\unitlength}{0.05 mm}%
  \begin{picture}(762.7, 591.4)(0,0)
  \put(0,0){\includegraphics{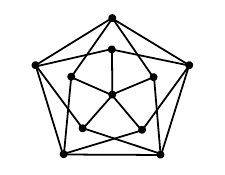}}
  \put(407.37,510.15){\fontsize{5.69}{6.83}\selectfont \makebox(60.0, 40.0)[l]{$u_1$\strut}}
  \put(662.68,356.13){\fontsize{5.69}{6.83}\selectfont \makebox(60.0, 40.0)[l]{$u_5$\strut}}
  \put(580.20,38.80){\fontsize{5.69}{6.83}\selectfont \makebox(60.0, 40.0)[l]{$u_4$\strut}}
  \put(135.92,46.31){\fontsize{5.69}{6.83}\selectfont \makebox(60.0, 40.0)[l]{$u_3$\strut}}
  \put(40.00,352.07){\fontsize{5.69}{6.83}\selectfont \makebox(60.0, 40.0)[l]{$u_2$\strut}}
  \put(353.73,446.32){\fontsize{5.69}{6.83}\selectfont \makebox(60.0, 40.0)[l]{$v_1$\strut}}
  \put(271.17,321.46){\fontsize{5.69}{6.83}\selectfont \makebox(60.0, 40.0)[l]{$v_2$\strut}}
  \put(522.74,340.08){\fontsize{5.69}{6.83}\selectfont \makebox(60.0, 40.0)[l]{$v_5$\strut}}
  \put(446.42,189.91){\fontsize{5.69}{6.83}\selectfont \makebox(60.0, 40.0)[l]{$v_4$\strut}}
  \put(258.39,190.17){\fontsize{5.69}{6.83}\selectfont \makebox(60.0, 40.0)[l]{$v_3$\strut}}
  \put(398.81,303.10){\fontsize{5.69}{6.83}\selectfont \makebox(20.0, 40.0)[l]{$w$\strut}}
  \end{picture}%
\else
  \setlength{\unitlength}{0.05 mm}%
  \begin{picture}(762.7, 591.4)(0,0)
  \put(0,0){\includegraphics{The-Grotzsch-graph}}
  \put(407.37,510.15){\fontsize{5.69}{6.83}\selectfont \makebox(60.0, 40.0)[l]{$u_1$\strut}}
  \put(662.68,356.13){\fontsize{5.69}{6.83}\selectfont \makebox(60.0, 40.0)[l]{$u_5$\strut}}
  \put(580.20,38.80){\fontsize{5.69}{6.83}\selectfont \makebox(60.0, 40.0)[l]{$u_4$\strut}}
  \put(135.92,46.31){\fontsize{5.69}{6.83}\selectfont \makebox(60.0, 40.0)[l]{$u_3$\strut}}
  \put(40.00,352.07){\fontsize{5.69}{6.83}\selectfont \makebox(60.0, 40.0)[l]{$u_2$\strut}}
  \put(353.73,446.32){\fontsize{5.69}{6.83}\selectfont \makebox(60.0, 40.0)[l]{$v_1$\strut}}
  \put(271.17,321.46){\fontsize{5.69}{6.83}\selectfont \makebox(60.0, 40.0)[l]{$v_2$\strut}}
  \put(522.74,340.08){\fontsize{5.69}{6.83}\selectfont \makebox(60.0, 40.0)[l]{$v_5$\strut}}
  \put(446.42,189.91){\fontsize{5.69}{6.83}\selectfont \makebox(60.0, 40.0)[l]{$v_4$\strut}}
  \put(258.39,190.17){\fontsize{5.69}{6.83}\selectfont \makebox(60.0, 40.0)[l]{$v_3$\strut}}
  \put(398.81,303.10){\fontsize{5.69}{6.83}\selectfont \makebox(20.0, 40.0)[l]{$w$\strut}}
  \end{picture}%
\fi
\caption[Figure 1]{The Gr\"{o}tzsch graph $\Gamma$. }\label{fig-1}
\end{figure}

Denote by $\mathcal{G}(n)$ the family of graphs obtained from $\Gamma$ by replacing each vertex $v_i$ (for $i=1, \cdots ,5$) and vertex $w$ with independent sets $V_i$ $(i=1, \cdots ,5)$ and $W$, respectively, such that $\sum_{i=1}^5|V_i|=\lfloor \frac{n-3}{2}  \rfloor$ (or $\lceil \frac{n-3}{2}  \rceil$), $|W|=\lceil \frac{n-7}{2}  \rceil$ (or $\lfloor \frac{n-7}{2}  \rfloor$). Two vertices in different independent sets are adjacent if and only if the corresponding original vertices in $\Gamma$ are adjacent (as shown in Figure 2). It is easy to check that each graph in $\mathcal{G}(n)$ is a triangle-free $4$-chromatic graph with $\lfloor\frac{(n-3)^2}{4}\rfloor+5$ edges.

\begin{figure}[H]
\centering
\ifpdf
  \setlength{\unitlength}{0.05 mm}%
  \begin{picture}(766.2, 592.8)(0,0)
  \put(0,0){\includegraphics{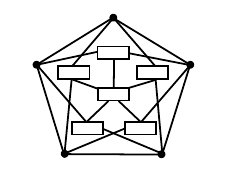}}
  \put(410.87,511.55){\fontsize{5.69}{6.83}\selectfont \makebox(60.0, 40.0)[l]{$u_1$\strut}}
  \put(666.19,357.53){\fontsize{5.69}{6.83}\selectfont \makebox(60.0, 40.0)[l]{$u_5$\strut}}
  \put(583.71,38.50){\fontsize{7.11}{8.54}\selectfont \makebox(75.0, 50.0)[l]{$u_4$\strut}}
  \put(139.28,44.12){\fontsize{5.69}{6.83}\selectfont \makebox(60.0, 40.0)[l]{$u_3$\strut}}
  \put(40.00,355.87){\fontsize{5.69}{6.83}\selectfont \makebox(60.0, 40.0)[l]{$u_2$\strut}}
  \put(358.95,444.06){\fontsize{4.55}{5.46}\selectfont \makebox(48.0, 32.0)[l]{$V_1$\strut}}
  \put(144.20,342.27){\fontsize{4.55}{5.46}\selectfont \makebox(48.0, 32.0)[l]{$V_2$\strut}}
  \put(569.48,341.04){\fontsize{4.55}{5.46}\selectfont \makebox(48.0, 32.0)[l]{$V_5$\strut}}
  \put(445.34,205.68){\fontsize{4.55}{5.46}\selectfont \makebox(48.0, 32.0)[l]{$V_4$\strut}}
  \put(265.33,205.24){\fontsize{4.55}{5.46}\selectfont \makebox(48.0, 32.0)[l]{$V_3$\strut}}
  \put(395.08,300.06){\fontsize{4.55}{5.46}\selectfont \makebox(16.0, 32.0)[l]{$W$\strut}}
  \end{picture}%
\else
  \setlength{\unitlength}{0.05 mm}%
  \begin{picture}(766.2, 592.8)(0,0)
  \put(0,0){\includegraphics{graphnew1}}
  \put(410.87,511.55){\fontsize{5.69}{6.83}\selectfont \makebox(60.0, 40.0)[l]{$u_1$\strut}}
  \put(666.19,357.53){\fontsize{5.69}{6.83}\selectfont \makebox(60.0, 40.0)[l]{$u_5$\strut}}
  \put(583.71,38.50){\fontsize{7.11}{8.54}\selectfont \makebox(75.0, 50.0)[l]{$u_4$\strut}}
  \put(139.28,44.12){\fontsize{5.69}{6.83}\selectfont \makebox(60.0, 40.0)[l]{$u_3$\strut}}
  \put(40.00,355.87){\fontsize{5.69}{6.83}\selectfont \makebox(60.0, 40.0)[l]{$u_2$\strut}}
  \put(358.95,444.06){\fontsize{4.55}{5.46}\selectfont \makebox(48.0, 32.0)[l]{$V_1$\strut}}
  \put(144.20,342.27){\fontsize{4.55}{5.46}\selectfont \makebox(48.0, 32.0)[l]{$V_2$\strut}}
  \put(569.48,341.04){\fontsize{4.55}{5.46}\selectfont \makebox(48.0, 32.0)[l]{$V_5$\strut}}
  \put(445.34,205.68){\fontsize{4.55}{5.46}\selectfont \makebox(48.0, 32.0)[l]{$V_4$\strut}}
  \put(265.33,205.24){\fontsize{4.55}{5.46}\selectfont \makebox(48.0, 32.0)[l]{$V_3$\strut}}
  \put(395.08,300.06){\fontsize{4.55}{5.46}\selectfont \makebox(16.0, 32.0)[l]{$W$\strut}}
  \end{picture}%
\fi
\caption[Figure 1]{The structure of graphs in $\mathcal{G}(n)$. }\label{fig-2}
\end{figure}

 The main result of this paper is the following theorem.

\begin{thm}\label{thm-1.8}
 Let $G$ be a graph on $n$ vertices with $n\geq 90$. If $G$ is triangle-free and $\chi(G)\geq 4$, then
 $$e(G)\leq \left\lfloor\frac{(n-3)^2}{4}\right\rfloor+5,$$
 with equality  if and only if $G\in \mathcal{G}(n)$ up to isomorphism.
 \end{thm}

Define
\[
d_2(G)=\min \left\{|T|\colon T\subseteq V(G),\ G-T \mbox{ is bipartite}\right\}.
\]

To prove Theorem~\ref{thm-1.8}, we need the following  vertex-stability result for Mantel's theorem.

\begin{thm}\label{thm-1.4}
 Let $G$ be a graph on $n$ vertices with $n\geq 90$. If $G$ is triangle-free and $d_2(G)\geq 4$, then
$$e(G)\leq  \left\lfloor\frac{(n-4)^2}{4}\right\rfloor+16.$$
 \end{thm}

 Let $H_n$ be a graph obtained from a cycle $v_1v_2v_3v_4v_5v_1$ by replacing the vertices $v_1,v_2,v_3$
with three independent sets of size four each, and replacing $v_4$ and $v_5$ with two independent sets of size $\lfloor\frac{n}{2}\rfloor-6$ and $\lceil\frac{n}{2}\rceil-6$, respectively, and two vertices in blow-up sets are adjacent if and only if the original vertices they replace are adjacent.
It is easy to see that  $H_n$ cannot be made bipartite by deleting fewer than 4 vertices, that is, $d_2(H_n)=4$
and $e(H_n)=\lfloor\frac{(n-4)^2}{4}\rfloor+16$. Thus the upper bound given in Theorem \ref{thm-1.4} is sharp.

\section{Some useful lemmas}
We start with  some definitions.
For $S\subseteq{V(G)}$, let $G[S]$ denote the subgraph of $G$ induced by $S$. Let $G-S$ denote the  subgraph induced by $V(G)\backslash{S}$. For simplicity, we write $E(S)$ and $e(S)$ for $E(G[S])$ and $e(G[S])$, respectively. Let $N_G(v)$ denote the set of neighbors of $v$ in $G$. For $S\subseteq V(G)$,  let $N_G(v,S)$ denote the set of neighbors of $v$ in $S$. Let $\deg_G(v)=|N_G(v)|$ and  $\deg_S(v)=|N_G(v,S)|$.  The {\it minimum degree} $\delta (G)$ of $G$ is defined as the minimum of $\deg_G(v)$ over all $v\in V(G)$.   For two disjoint sets $S,T\subseteq V(G)$, let $G[S,T]$ denote the subgraph of $G$ with  vertex set $S\cup T$ and  edge set $\left\{xy\in E(G)\colon x\in S,\ y\in T\right\}$.  Let $e_G(S,T)=e(G[S,T])$.  If the context is clear, we often omit the subscript $G$.

In 1982, H{\"a}ggkvist \cite{Haggkvist} determined the structure of triangle-free graphs on $n$ vertices with minimum degree greater than $\frac{3n}{8}$.

\begin{thm}[\cite{Haggkvist}]\label{thm-1.5}
 Let $G$ be a triangle-free graph on $n$ vertices. If $\delta(G)>\frac{3}{8}n$, then $G$ is either a bipartite graph or a subgraph of a blow-up of $C_5$.
\end{thm}

\begin{lem}\label{LEMMA:matrix-d-regular}
Suppose that $z_1,z_2,\ldots,z_5$ are positive reals satisfying $z_1+z_2+\cdots+z_5=z$ and $\min_{1\leq i\leq 5}z_i \geq z_0$ for some  $z_0 \ge 0$.
    Then
    \begin{align*}
        z_1z_3+z_2z_4+ z_3z_5+z_4z_1+z_5z_2
        \ge  (2z -  5z_0)z_0.
    \end{align*}
\end{lem}

\begin{proof}
Let $z_i=z_0+\xi_i$, $i=1,2,\ldots,5$. Then
\begin{align*}
&2(z_1z_3+z_2z_4+ z_3z_5+z_4z_1+z_5z_2)\\[3pt]
=&z_1(z_3+z_4)+z_2(z_4+z_5)+z_3(z_1+z_5)+z_4(z_1+z_2)+z_5(z_2+z_3)\\[3pt]
=&\xi_1(z_3+z_4)+\xi_2(z_4+z_5)+\xi_3(z_1+z_5)+\xi_4(z_1+z_2)+\xi_5(z_2+z_3)+2z_0z\\[3pt]
\geq &2(\xi_1+\xi_2+\cdots+\xi_5)z_0+2z_0z\\[3pt]
= & 2(z-5z_0)z_0+2z_0z\\[3pt]
=&2(2z-5z_0)z_0.
\end{align*}
\end{proof}

Now we prove the following lemma, which will be used  to prove Theorem \ref{thm-1.4}.

\begin{lem}\label{lem-2.3}
Let $G$ be a triangle-free  graph on $n$ vertices with $e(G)\geq \frac{(n-4)^2}{4}+16$ and $n\geq 90$. Then there exists $T\subseteq V(G)$ with $|T| \leq 15$ such that $G-T$ is a bipartite graph. Moreover, for any $S\subset T$,
\begin{align}\label{ineq-key1000}
\frac{(n-|S|-4)^2}{4}+\sum_{v\in S} \deg(v) \leq  \frac{(n-4)^2}{4}.
\end{align}
\end{lem}
\begin{proof}
We obtain an induced bipartite subgraph from $G$ by a standard vertex-deletion  procedure. Set $G_0=G$. Let us start with $i=0$ and obtain a sequence of graphs $G_1,G_2,\ldots,G_\ell,\ldots$ as follows: In the $i$th step, if there exists $v_{i+1}\in V(G_i)$ satisfying  $\deg_{ G_i}(v_{i+1})\leq \frac{3}{8}(n-i)$ then let $G_{i+1}$ be the graph obtained by deleting $v_{i+1}$ from $G_i$, and proceed to the $(i+1)$-th step. Otherwise we stop and set $G^*=G_i$ and $T=\{v_1,v_2,\ldots,v_i\}$.  Clearly, the procedure will terminate after finitely many steps.
 Suppose to the contrary that $|T|\geq 16$.  Since $G_{16}$ is triangle-free, by  Theorem \ref{thm-1.1} we have $e(G_{16})\leq \frac{(n-16)^2}{4}$. It follows that
\begin{align*}
e(G)\leq e(G_{16})+\sum_{i=0}^{15} \frac{3}{8}(n-i)
&\leq\frac{(n-16)^2}{4}+ \frac{3}{16}(2n-15)\times 16\leq \frac{(n-4)^2}{4}+15,
\end{align*}
a contradiction. Thus $|T|\leq 15$.

For each $v_j\in T$, we have
\[
\deg(v_j)=\deg_{G_{j-1}}(v_j)+\deg_{T}(v_j)\leq \frac{3}{8}(n-j+1)+j-1=\frac{3}{8}n+\frac{5}{8}(j-1).
\]
Thus for any $S\subset T$,
\begin{align*}
\sum_{v\in S} \deg(v) \leq \sum_{0\leq j\leq |S|-1} \left(\frac{3}{8}n+\frac{5}{8}(14-j)\right) = \frac{3}{8}n|S| +\frac{5}{16} |S|(29-|S|).
\end{align*}
Note that
\begin{align*}
\frac{(n-|S|-4)^2}{4}+\sum_{v\in S} \deg(v)&\leq \frac{(n-4)^2}{4} +\frac{1}{4} |S| ( |S|- 2 n + 8 )+ \frac{3}{8}n|S| +\frac{5}{16} |S|(29-|S|)\\[3pt]
&=\frac{(n-4)^2}{4}-|S|\left(\frac{n}{8}+\frac{|S|}{16}-\frac{177}{16}\right)\\[3pt]
&\leq \frac{(n-4)^2}{4},
\end{align*}
where the last inequality holds  for $n\geq 90$.
Thus  \eqref{ineq-key1000} holds.

Note that  $\delta(G^*)\geq \frac{3}{8}(n-|T|)$. By Theorem \ref{thm-1.5} we conclude that  $G^*$ is either bipartite or a subgraph of blow-up of $C_5$. Suppose that the latter case holds.  Let $D_1,D_2,\ldots,D_5$ be the partition of $V(G^*)$ such that each $D_i$ is an independent set of $G^*$ and  each edge of $G^*$ has one endpoint in  $D_i$ and the other in $D_{i+1}$ for some $i\in [5]$, where indices are taken modulo 5.

Let $t=|T|\leq 15$. Note that for $n\geq 45$,
\begin{align*}
e(G^*)\geq e(G)-\sum_{i=1}^{t}\deg_{G_i}(v_i)&\geq \frac{(n-4)^2}{4}+16-\sum_{i=0}^{t-1}\frac{3}{8}(n-i)\\[3pt]
&=\frac{(n-t)^2}{4}-\frac{(16-t)n}{8}-\frac{t(t+3)}{16}+20\\[3pt]
&>\frac{13(n-t)^2}{64}+\frac{3(n-t)^2}{64}-\frac{(16-t)n}{8}\\[3pt]
&>\frac{13(n-t)^2}{64}
\end{align*}
and $\sum_{i\in [5]}|D_i|=n-t$.
It follows that
\begin{align}\label{ineq-2.1}
\sum_{i\in [5]}|D_i||D_{i+1}| \geq e(G^*)>\frac{13}{64}(n-t)^2.
\end{align}
Since the minimum degree is at least $\frac{3}{8}(n-t)$ in $G^*$, we infer that
\begin{align}\label{ineq-2.2}
|D_{i-1}| + |D_{i+1}|\geq \frac{3}{8}(n-t)
     \quad\text{for each}\quad i=1,2,3,4,5.
    \end{align}
    Let $z_i \coloneqq |D_{i-1}| + |D_{i+1}|$, $i=1,2,3,4,5$.
 Since $\sum_{1\leq i\leq 5} z_i=2(n-t)$, applying Lemma~\ref{LEMMA:matrix-d-regular}  with $z_0=\frac{3}{8}(n-t)$ and $z=2(n-t)$ we obtain that
    \begin{align}\label{equ:opt1-e}
       \sum_{1\leq i\leq 5} z_iz_{i+2}
        \ge  (2z - 5z_0)z_0 =\frac{51}{64}(n-t)^2.
    \end{align}
It follows that
\begin{align*}
\sum_{i\in [5]}|D_i||D_{i+1}|
        = \left(\sum_{i\in [5]}|D_i|\right)^2-\sum_{i\in [5]} z_i z_{i+2}
        & = \left(n-t\right)^2- \sum_{i\in [5]} z_i z_{i+2}\leq \frac{13}{64}(n-t)^2,
\end{align*}
which contradicts~\eqref{ineq-2.1}.   Thus $G^*=G-T$ is a bipartite graph. This completes the proof of Lemma \ref{lem-2.3}.
\end{proof}
For a graph $G$, a collection of vertex-disjoint edges in $G$ is called a {\it matching} of $G$. The {\it matching number}, denoted by $\nu (G)$, is defined as the number of edges of  a maximum matching in $G$. A {\it covering} of $G$ is a subset $K$ of $V$ such that every edge of $G$ has an endpoint in $K$. The {\it covering number}, denoted by $\tau(G)$, is defined as the  number of vertices of  a minimum covering in $G$. Denote by $G\cup H$ the vertex-disjoint union of $G$ and $H$.

The next lemma determines the structure of a triangle-free graph $H$ with $\tau(H)\geq 4$ and $\nu(H)=3$.

\begin{lem}\label{lem-3.1}
Let $H$ be a triangle-free graph without isolated vertices.  If $\nu(H)=3$ and $\tau(H)\geq 4$, then $H$ is either $(i)$ $H$ is a $C_7$ or $(ii)$  $|V(H)|=r+6$ and $H$ contains a $C_5\cup K_{1,r}$ as a subgraph.
\end{lem}
\begin{proof}
Since $\tau(H)>\nu(H)$, we infer that $H$ is non-bipartite. Let $C$ be the shortest odd cycle in $H$ and let $\ell$ be the length of  $C$. Since $H$ is triangle-free and $\nu(H)=3$, we have $5\leq \ell\leq 7$. If $\ell=7$, then  $H$ has to be a $C_7$ since $\nu(H)=3$.  Thus we may assume that $\ell=5$.

Let $C=v_1v_2v_3v_4v_5v_1$.  If there is an edge $xy$ in $V(H)-V(C)$, then by $\nu(H)=3$ we infer that at most one of $x$ and $y$ has neighbors outside $V(C)\cup \{x,y\}$. Then $V(H) = V(C) \cup \{x,y\}\cup N(x)\cup N(y)$. This implies that $H$ contains a $C_5\cup K_{1,r}$ and $|V(H)|=r+6$, we are done. Therefore,  $V(H)-V(C)$  is an independent set and  $V(C)$ is a covering of $H$. By $\tau(H)\geq 4$, $\{v_1,v_3,v_5\}$ is not a covering. We infer that at least one of $v_2$, $v_4$ has neighbors outside $C$. Without loss of generality, assume $v_2a\in E(H)$. Similarly, since $\{v_2,v_4,v_5\}$ is not a covering, at least one of $v_1$, $v_3$ has neighbors outside $C$. Assume $v_1b\in E(H)$. Since $H$ is triangle-free, we have $a\neq b$. As  $\{v_1,v_2,v_4\}$ is not a covering, we may assume $v_3c\in E(H)$. Since $\nu(H)=3$ and $H$ is triangle-free, we must have $c=b$. Now $v_1bv_3v_4v_5v_1$ is a new $C_5$ and $v_2a$ is an edge outside of this $C_5$. By the previous case we are done.
\end{proof}

\section{Proof of Theorem \ref{thm-1.4} }

\begin{proof}[Proof of Theorem \ref{thm-1.4}]
Suppose for contradiction that
\begin{align}\label{ineq-assumpt}
e(G)> \frac{(n-4)^2}{4}+16.
\end{align}
By Lemma \ref{lem-2.3}, there exists $T\subseteq V(G)$ with $|T| \leq 15$ such that $G-T$ is bipartite on partite sets $X,Y$. Let
\begin{align}\label{eq-0}
T_X=\{v\in T:\deg(v,Y)\geq  \deg(v,X)\} \mbox{ and }T_Y=\{v\in T:\deg(v,X)>  \deg(v, Y)\} .
\end{align}
Clearly $T=T_X\cup T_Y$ and $(X\cup T_X,Y\cup T_Y)$ is a bipartition of $V(G)$. Let $X^*=X\cup T_X$, $Y^*=Y\cup T_Y$ and let $H$ be the subgraph of $G$ induced by the edge set $E(X^*)\cup E(Y^*)$. Thus, $H$ contains no isolated vertices and $e(G)=e(H)+e(X^*,Y^*)$.

Note that  $\tau(H)\geq d_2(G)\geq 4$. Since $H$ is triangle-free,  it is known that $\tau(H) < 2\nu(H)$. It follows that $\nu(H)>\frac{\tau(H)}{2}\geq 2$. Then $\nu(H)\geq 3$. Now we distinguish two cases.

\vspace{5pt}
 {\bf \noindent Case 1. } $\nu(H)=3$.
\vspace{5pt}

By Lemma \ref{lem-3.1}, $H$ is either a $C_7$ or contains a $C_5\cup K_{1,r}$ as a subgraph. If $H$  is a $C_7$, by symmetry we assume $V(H)\subset X^*$. Since $G$ is triangle-free,
  each vertex in $Y^*$ has at most three neighbors on $V(H)$. It implies that  $e(V(H),Y^*)\leq 3|Y^*|$. Then
\begin{align*}
e(X^*, Y^*)= e(X^*\setminus V(H),Y^*)+e(V(H),Y^*)&\leq (|X^*|-7)|Y^*|+3|Y^*|\\[3pt]
&\leq \frac{(|X^*|+|Y^*|-4)^2}{4}\\[3pt]
&= \frac{(n-4)^2}{4}.
\end{align*}
As $e(H)=7$, we have
\begin{align*}
e(G)= e(X^*, Y^*)+e(H)&\leq \frac{(n-4)^2}{4}+7< \left\lfloor\frac{(n-4)^2}{4}\right\rfloor+16,
\end{align*}
contradicting \eqref{ineq-assumpt}. Thus $H$ contains a $C_5\cup K_{1,r}$ as a subgraph and $|V(H)|=6+r$.

Let $C$ be the  vertex set  of the $C_5$ and $U$ be the vertex set of the $K_{1,r}$.    If $V(K_{1,r})\cap (X\cup Y)\neq \emptyset $, then choose  $v\in T\cap V(K_{1,r})$ and let $V(K_{1,r})\cap(X\cap Y)=\{v_1,v_2,\ldots,v_\ell\}$.  If $V(K_{1,r})\subset T$, then let  $\ell=1$ and let $vv_1$ be an edge in $K_{1,r}$. Clearly $v,v_1,v_2,\ldots,v_\ell$ form a subtree $K_{1,\ell}$ of $K_{1,r}$ with root $v$.

Let $S= C\cup\{v,v_1,\ldots,v_\ell\}$, $X_1= X \cup  (S\cap T_X)$ and $Y_1= Y\cup (S\cap T_Y)$. Then
\begin{align}\label{ineq-7}
e(G) \leq  \sum_{v\in {T\setminus S}}\deg(v)+e(X_1,Y_1) + e_H(S).
\end{align}

By symmetry we distinguish two subcases. Note that by the definition of $T_X,T_Y$ in \eqref{eq-0}, $T_X$ and $T_Y$ are not symmetric. However in the following proof we shall only use $\deg(v,X)\geq \deg(v, Y)$ for each $v\in T_Y$ and $\deg(u,Y)\geq \deg(u, X)$ for each $u\in T_X$.

\vspace{5pt}
 {\bf \noindent Subcase 1.1.}   $C\subset X_1$ and $\{v,v_1,\ldots,v_\ell\}\subset Y_1$.

Recall that $v$ be the root of $K_{1,r}$ and $v\in T$. Then $v\in T_Y$. Since $G$ is triangle-free, each vertex in $Y_1$ has at most two neighbors in $C$. Since $G$ is triangle-free and $\{v,v_1\}$ is an edge, each vertex in $X_1$ has at most one neighbor in $\{v,v_1\}$.  If $\ell=1$ and $v_1\in T$, then
 \begin{align}\label{ineq-9-0}
 e(X_1,Y_1)+e_H(S)\leq |X_1||Y_1| -3|Y_1|-|X_1\setminus C|+6 = (|X_1|-3)(|Y_1|-1)+8.
 \end{align}

If $v_1\in Y$, then let $k$ be the number of neighbors of $v$ in $X_1\setminus C$. By the definition of $T_y$, $k\geq 1$. Since $G$ is triangle-free, there is no edge between $N(v,X_1\setminus C)$ and $\{v_1,v_2,\ldots,v_\ell\}$. Thus,
\begin{align}\label{ineq-9}
 e(X_1,Y_1)+e_H(S)&\leq |X_1||Y_1|-\ell k -3|Y_1|-(|X_1|-k-5)+(5+j)\nonumber\\[3pt]
 &=(|X_1|-3)(|Y_1| -1)-(\ell-1)(k-1)+8\nonumber\\[3pt]
 &< \frac{(n-|T\setminus S|-4)^2}{4}+8.
\end{align}
Comparing \eqref{ineq-9-0} and \eqref{ineq-9},  we have
\begin{align*}
 e(X_1,Y_1)+e_H(S)&\leq \frac{(n-|T\setminus S|-4)^2}{4}+8.
\end{align*}
Together this with \eqref{ineq-7}  and \eqref{ineq-key1000}, we get
 \begin{align*}
 e(G)\leq \frac{(n-|T\setminus S|-4)^2}{4} +\sum_{v\in {T\setminus S}}\deg(v)+8 \leq \frac{(n-4)^2}{4}+8,
 \end{align*}
contradicting \eqref{ineq-assumpt}.

\vspace{5pt}
 {\bf \noindent Subcase 1.2.}  $C \cup \{v,v_1,\ldots,v_\ell\}\subset X_1$.

 If $\ell=1$,  noting that each of $v$  and $v_1$ has at most two neighbors in $C$,  then $e_H(S)\leq 5+1+2\times2=10$. Since $G$ is triangle-free, each vertex in $Y_1$ has at most
 three neighbors in $C\cup \{v,v_1\}$.  Then
 \begin{align}\label{ineq-new1}
 e(X_1,Y_1) + e_H(S)\leq  |X_1||Y_1| -4|Y_1|+10\leq \frac{(n-|T\setminus S|-4)^2}{4}+10.
 \end{align}

If $\ell\geq 2$, then by  $\nu(H)=3$ each $v_i$ has no neighbor in $C$. As $v$ has at most two neighbors on $C$,  $e_H(S)\leq 5+\ell+2=\ell+7$.  Note that $v\in T_X$. By the definition of $T_X$, $v$ has at least $\ell$ neighbors in $Y$. Note that each vertex in $Y_1$ has at most
 three neighbors on $C\cup \{v,v_1\}$. Moreover, there is no edge between $\{v_1,v_2,\ldots,v_\ell\}$ and $N(v, Y)$. It follows that
 \begin{align}\label{ineq-new2}
 e(X_1,Y_1)+ e_H(S)& \leq |X_1||Y_1| -4|Y_1|-(\ell-1)\ell+\ell+7\nonumber\\[3pt]
  &\leq  (|X_1|-4)|Y_1|+7\nonumber\\[3pt]
 &\leq \frac{(n-|T\setminus S|-4)^2}{4}+7.
 \end{align}
 Combining \eqref{ineq-new1} and \eqref{ineq-new2}, we infer that
 \begin{align*}
e(G) \leq e(X_1,Y_1)+ e_H(S) +\sum_{v\in {T\setminus S}}\deg(v) &\leq  \frac{(n-|T\setminus S|-4)^2}{4}+10+\sum_{v\in {T\setminus S}}\deg(v) \\[3pt]
&\overset{\eqref{ineq-key1000}}{\leq} \frac{(n-4)^2}{4}+8,
 \end{align*}
contradicting \eqref{ineq-assumpt}.

 \vspace{5pt}
 {\bf \noindent Case 2. } $\nu(H)\geq 4$.
 \vspace{5pt}

Let us choose a matching $M$ of size 4 in $H$.  Let $M_1$ be the set of edges of $M$ in $G[X^*]$ and $M_2$ be the set of  edges of $M$ in $G[Y^*]$.  Set $S_1=V(M_1)$, $S_2=V(M_2)$, $S=S_1\cup S_2$, $X_1=X\cup S_1$ and $Y_1=Y\cup S_2$.  Let $H_1=G[X_1]\cup G[Y_1]$.
 Then
\begin{align}\label{ineq-2}
e(G) \leq  e_G(X_1,Y_1) + e(H_1)+\sum_{v\in {T\setminus S}}\deg(v).
\end{align}

Note that for each $uv\in M_1$, $u$ and $v$ have no common neighbors in $Y_1$. Similarly, for each $uv\in M_2$, $u$ and $v$ have no common neighbors in $X_1$. Let $G^*$ be the graph obtained from $G[X_1,Y_1]$ by contracting each edge in $M$ to a single vertex. Let $\bar{e}(X\setminus S,Y\setminus S)$ be the number of missing edges in $G$ between $X\setminus S$ and $Y\setminus S$. Then
\begin{align*}
e_{G}(X_1,Y_1)=e(G^*) \leq \frac{(n-|T\setminus S|-4)^2}{4}-\bar{e}(X\setminus S,Y\setminus S).
\end{align*}
Using \eqref{ineq-2} and \eqref{ineq-key1000}, we get
\begin{align}\label{ineq-4}
e(G) \leq  e_{G}(X_1,Y_1)+\sum_{v\in {T\setminus S}}\deg(v)+ e(H_1)\leq \frac{(n-4)^2}{4}+e(H_1)-\bar{e}(X\setminus S,Y\setminus S).
\end{align}
 By \eqref{ineq-assumpt} we have
 \begin{align}\label{ineq-3.10}
     e(H_1)\geq 17+\bar{e}(X\setminus S,Y\setminus S).
 \end{align}

Note that $M$ is a matching of size $4$ in $H$.  Among all such matchings, we choose $M$ such that $p=|V(M)\cap (X\cup Y)|$ is maximum. Clearly $p\leq 4$. Assume that $M=\{u_1v_1,\cdots, u_4v_4\}$ with $v_1,\ldots,v_p\in X\cup Y$ and $u_1,\ldots,u_4,v_{p+1},\ldots,v_4\in T$.  Since $p$ is maximum, we infer that $u_{p+1},\ldots,u_4$, $v_{p+1},\ldots,v_4$ have no neighbor in $(X\cup Y)\setminus S$ in $H_1$.
By symmetry, assume that $u_1\in S_1$  and $a=|N_{H_1}(u_1,X\setminus S)|$ is the maximum over all $|N_{H_1}(u_i,X\cup Y\setminus S)|$, $i=1,2,3,4$.
  Thus,
\begin{align}\label{ineq-eh1}
e(H_1) = \sum_{1\leq i\leq p} |N_{H_1}(u_i,(X\cup Y)\setminus S)| +e_{H_1}(S).
\end{align}

If $p=0$ then $S\subset T$. By Mantel's theorem $e(H_1)\leq \frac{|S|^2}{4}=  16$, a contradiction. Thus we have $1\leq p\leq 4$.
  Recall that  $u_1\in S_1$ and $a=|N_{H_1}(u_1,X\setminus S)|$ is the maximum over all $|N_{H_1}(u_i,X\cup Y\setminus S)|$, $i=1,2,3,4$. Let $|M_1|=q$ and $|M_2|=4-q$.
  Since $G$ is triangle-free, by Mantel's theorem we have
  \[
  e_{H_1}(S)=e_{H_1}(S_1)+e_{H_1}(S_2)\leq q^2+(4-q)^2.
  \]
By \eqref{ineq-eh1} we get
 \begin{align}\label{eq-3.17}
 e(H_1) \leq pa+q^2+(4-q)^2.
 \end{align}
 Note that \eqref{ineq-3.10} implies $a\geq 1$.

\vspace{5pt}
 {\bf \noindent Subcase 2.1.} $q=1$.

By \eqref{ineq-3.10} and \eqref{eq-3.17},
\[
17+\bar{e}(X\setminus S,Y\setminus S)\leq e(H_1)\leq pa+q^2+(4-q)^2 \leq 4a+10.
\]
It follows that $a\geq 2$.
 Without loss of generality, let $u_2\in S_2$ such that  $b=|N_{H_1}(u_2,Y\setminus S)|$ is the maximum over all $|N_{H_1}(u_i,Y\setminus S)|$, $i=2,3,4$. Clearly $b\leq a$.

If $b=a$, then by the definition of $T_Y$, $|N(u_2,X\setminus S)|\geq |N(u_2,Y\setminus S)|=a$. Since there is no edge between $N(u_2,X\setminus S)$ and $N(u_2,Y\setminus S)$, we infer that
  \[
 4a+10\geq  17+\bar{e}(X\setminus S,Y\setminus S) \geq 17+a^2,
  \]
  a contradiction. Thus $b\leq a-1$.

 By \eqref{ineq-eh1}, $e(H_1)\leq a+3b+10\leq 4a+7$.   Recall that $u_1$ has at least $a+1-3=a-2$ neighbors in $Y\setminus S$. Since there is no edge between $N(u_1,X\setminus S)$ and $N(u_1,Y\setminus S)$, we infer that
  \[
  \ 4a+7\geq e(H_1)\geq  17+\bar{e}(X\setminus S,Y\setminus S) \geq 17+ a(a-2).
\]
It follows that $a^2-6a+10\leq 0$, which is impossible.

  \vspace{5pt}
 {\bf \noindent Subcase 2.2.} $2\leq q\leq 3$.

 By \eqref{eq-3.17} and \eqref{ineq-3.10}, we have
 \[
 17+\bar{e}(X\setminus S,Y\setminus S)\leq e(H_1) \leq  pa+q^2+(4-q)^2\leq 4a+10.
 \]
 It follows that $a\geq 2\geq 4-q$. Note that $u_1$ has at least $a+1-(4-q)=a+q-3$ neighbors in $Y\setminus S$. Since  there is no edges between $N(u_1,X\setminus S)$ and $N(u_1,Y\setminus S)$, we infer that
  \[
\bar{e}(X\setminus S,Y\setminus S) \geq a(a+q-3).
  \]
  It follows that
  \[
   a(a+q-3)+17\leq 4a+q^2+(4-q)^2 =2q^2-8q+16+4a.
  \]
  Thus  we have
  \[
  2q^2-8q+16+4a - a(a+q-3)-17 = 2q^2-(8+a)q+7a-a^2-1\geq 0.
  \]
  Define $f(q):=2q^2-(8+a)q+7a-a^2-1$.  Note that
 \[
 f(2)=-9-a(a-5)\leq -3, \   \  f(3)=-7-a(a-4)\leq -3,
 \]
 a contradiction.

 \vspace{5pt}
 {\bf \noindent Subcase 2.3.}  $q=4$.

By \eqref{ineq-eh1} we have $e(H_1)\leq e_{H_1}(S)+pa$. Since $u_1$ has at least $a+1$ neighbors in $Y\setminus S$ and $e(N(u_1,X\setminus S),N(u_1,Y\setminus S))=0$,
\[
 \bar{e}(X\setminus S,Y\setminus S) \geq a(a+1).
  \]
 By  \eqref{ineq-3.10},
 \[
 17+a(a+1)\leq 17+\bar{e}(X\setminus S,Y\setminus S)\leq e(H_1) \leq  pa+e_{H_1}(S).
 \]
 Thus $e_{H_1}(S) \geq 17+a(a+1)-pa$.
 By Mantel's theorem $e_{H_1}(S)\leq 16$. Thus $p\geq a+2\geq 3$ and thereby
 \[
 e_{H_1}(S) \geq 17+a(a+1)-4a\geq (a-1)(a-2)+15.
 \]
 Thus $e_{H_1}(S)\geq 15$.

 Since $e_{H_1}(S)\geq 15$ and $G[S]$ is triangle-free, by Theorem \ref{erdos2}, $G[S]$ must be bipartite. Since $\{u_1v_1,u_2v_2,u_3v_3,u_4v_4\}$ is a matching, $G[S]$ is a bipartite graph with partite sets of equal size. Note that  $p\geq 3$. It follows that $\{v_1,v_2,v_3\}$ is in one partite set and $\{u_1, u_2, u_3\}$ is in the other partite set in $G[S]$.
Since $H_1[S]$ is obtained from a $K_{4,4}$ by removing at most one edge, at least one of  $u_1v_2$ and  $u_1v_3$ is an edge. Thus $u_1$ has at least $a+2$ neighbors in $X$ and  thereby has at least $a+2$ neighbors in $Y$. Therefore,
\[
 \bar{e}(X\setminus S,Y\setminus S) \geq a(a+2).
  \]
  By \eqref{ineq-assumpt} and \eqref{ineq-key1000},
  \[
  a(a+2) +17 \leq  \bar{e}(X\setminus S,Y\setminus S) +17 \leq  e(H_1) \leq  pa+e_{H_1}(S).
  \]
  Then
  \[
 16\geq  e_{H_1}(S) \geq a(a+2)+17-pa\geq a(a+2)+17-4a =(a-1)^2+16.
  \]
Thus $a=1$, $p=4$ and $e_{H_1}(S)=16$. Hence $u_iv_j\in E(G)$ for all $i=1,2,3,4$ and $j=1,2,3,4$. Then $u_1$ has exactly $5$ neighbors in $X$ and  thereby has at least $5$ neighbors in $Y$.  It follows that
  \[
  1\times 5+17 \leq  \bar{e}(X\setminus S,Y\setminus S) +17 \leq  e(H_1) \leq  pa+e_{H_1}(S)= 4+16,
  \]
  the final contradiction.
\end{proof}
\section{Proof of Theorem \ref{thm-1.8} }
\begin{proof}[Proof of Theorem \ref{thm-1.8}]
 Let $G$ be a triangle-free  graph on $n$ vertices with $\chi(G)\geq 4$ and $n\geq 90$.
If $d_2(G)\geq 4$, then by Theorem \ref{thm-1.4} and $n\geq 90$ we have
\[
e(G)\leq\left\lfloor\frac{(n-4)^2}{4}\right\rfloor+16< \left\lfloor\frac{(n-3)^2}{4}\right\rfloor-\frac{n}{2}+16<\left\lfloor\frac{(n-3)^2}{4}\right\rfloor+4
\]
and we are done. Note that $d_2(G)\leq 1$ would imply $\chi(G)\leq 3$. Thus we may assume that $2\leq d_2(G)\leq 3$.

If $d_2(G)= 2$, then there exist $x,y\in V(G)$ such that  $G-\{x,y\}$ is bipartite.
Let $A,B$ be partite sets of $G-\{x,y\}$.
 Since $\chi(G)\geq 4$, $\chi(G[\{x,y\}\cup A])=3$. Thus $xy\in E(G)$ and $G[\{x,y\}\cup A]$ contains an odd cycle $C$.
Since $A$ is an independent set,  $C$ has to be a triangle, a contradiction. Thus $d_2(G)= 3$.

Let $x,y,z\in V(G)$ such that $G-\{x,y,z\}$ is a bipartite graph on partite sets $A, B$. Since $\chi(G)\geq 4$, $\chi(G[\{x,y,z\}\cup A]) \geq 3$ and $\chi(G[\{x,y,z\}\cup B]) \geq 3$.
Thus, both $G[\{x,y,z\}\cup A]$ and $G[\{x,y,z\}\cup B]$ contain  odd cycles.

Let $C$ be an odd cycle in $G[\{x,y,z\}\cup A]$.
Since $A$ is an independent set,  at most $\frac{|V(C)|-1}{2}$ vertices on $C$ are in $A$. It follows that at least $\frac{|V(C)|+1}{2}$ vertices are in $\{x,y,z\}$. It follows that $|V(C)|\leq 5$. Since $G$ is triangle-free, $|V(C)|=5$ and $x,y,z$ are all on $C$.
Moreover,  there is exactly one edge in  $G[\{x,y,z\}]$. Without loss of generality, assume that $xy\in E(G)$ and $xz,yz\notin E(G)$. By the same argument, one can show that $G[\{x,y,z\}\cup B]$ contains a $C_5$ as well.

\begin{figure}[H]
\centering
\ifpdf
  \setlength{\unitlength}{0.05 mm}%
  \begin{picture}(781.2, 700.5)(0,0)
  \put(0,0){\includegraphics{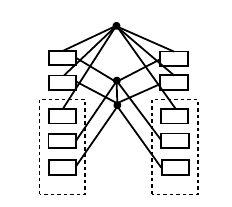}}
  \put(415.40,609.00){\fontsize{7.11}{8.54}\selectfont \makebox(25.0, 50.0)[l]{$z$\strut}}
  \put(162.75,482.07){\fontsize{5.69}{6.83}\selectfont \makebox(120.0, 40.0)[l]{$A_{xz}$\strut}}
  \put(541.29,479.85){\fontsize{5.69}{6.83}\selectfont \makebox(120.0, 40.0)[l]{$B_{xz}$\strut}}
  \put(162.07,401.49){\fontsize{5.69}{6.83}\selectfont \makebox(120.0, 40.0)[l]{$A_{yz}$\strut}}
  \put(541.21,400.21){\fontsize{5.69}{6.83}\selectfont \makebox(120.0, 40.0)[l]{$B_{yz}$\strut}}
  \put(401.35,436.38){\fontsize{7.11}{8.54}\selectfont \makebox(25.0, 50.0)[l]{$x$\strut}}
  \put(424.10,320.62){\fontsize{7.11}{8.54}\selectfont \makebox(25.0, 50.0)[l]{$y$\strut}}
  \put(184.66,283.07){\fontsize{5.69}{6.83}\selectfont \makebox(60.0, 40.0)[l]{$A_z$\strut}}
  \put(180.44,200.57){\fontsize{5.69}{6.83}\selectfont \makebox(60.0, 40.0)[l]{$A_x$\strut}}
  \put(183.96,112.83){\fontsize{5.69}{6.83}\selectfont \makebox(60.0, 40.0)[l]{$A_y$\strut}}
  \put(561.26,282.76){\fontsize{5.69}{6.83}\selectfont \makebox(60.0, 40.0)[l]{$B_z$\strut}}
  \put(564.22,199.76){\fontsize{5.69}{6.83}\selectfont \makebox(60.0, 40.0)[l]{$B_x$\strut}}
  \put(566.94,113.60){\fontsize{5.69}{6.83}\selectfont \makebox(60.0, 40.0)[l]{$B_y$\strut}}
  \put(40.00,175.06){\fontsize{5.69}{6.83}\selectfont \makebox(40.0, 40.0)[l]{$A'$\strut}}
  \put(701.16,178.08){\fontsize{5.69}{6.83}\selectfont \makebox(40.0, 40.0)[l]{$B'$\strut}}
  \end{picture}%
\else
  \setlength{\unitlength}{0.05 mm}%
  \begin{picture}(781.2, 700.5)(0,0)
  \put(0,0){\includegraphics{Graph6}}
  \put(415.40,609.00){\fontsize{7.11}{8.54}\selectfont \makebox(25.0, 50.0)[l]{$z$\strut}}
  \put(162.75,482.07){\fontsize{5.69}{6.83}\selectfont \makebox(120.0, 40.0)[l]{$A_{xz}$\strut}}
  \put(541.29,479.85){\fontsize{5.69}{6.83}\selectfont \makebox(120.0, 40.0)[l]{$B_{xz}$\strut}}
  \put(162.07,401.49){\fontsize{5.69}{6.83}\selectfont \makebox(120.0, 40.0)[l]{$A_{yz}$\strut}}
  \put(541.21,400.21){\fontsize{5.69}{6.83}\selectfont \makebox(120.0, 40.0)[l]{$B_{yz}$\strut}}
  \put(401.35,436.38){\fontsize{7.11}{8.54}\selectfont \makebox(25.0, 50.0)[l]{$x$\strut}}
  \put(424.10,320.62){\fontsize{7.11}{8.54}\selectfont \makebox(25.0, 50.0)[l]{$y$\strut}}
  \put(184.66,283.07){\fontsize{5.69}{6.83}\selectfont \makebox(60.0, 40.0)[l]{$A_z$\strut}}
  \put(180.44,200.57){\fontsize{5.69}{6.83}\selectfont \makebox(60.0, 40.0)[l]{$A_x$\strut}}
  \put(183.96,112.83){\fontsize{5.69}{6.83}\selectfont \makebox(60.0, 40.0)[l]{$A_y$\strut}}
  \put(561.26,282.76){\fontsize{5.69}{6.83}\selectfont \makebox(60.0, 40.0)[l]{$B_z$\strut}}
  \put(564.22,199.76){\fontsize{5.69}{6.83}\selectfont \makebox(60.0, 40.0)[l]{$B_x$\strut}}
  \put(566.94,113.60){\fontsize{5.69}{6.83}\selectfont \makebox(60.0, 40.0)[l]{$B_y$\strut}}
  \put(40.00,175.06){\fontsize{5.69}{6.83}\selectfont \makebox(40.0, 40.0)[l]{$A'$\strut}}
  \put(701.16,178.08){\fontsize{5.69}{6.83}\selectfont \makebox(40.0, 40.0)[l]{$B'$\strut}}
  \end{picture}%
  \fi
\caption[Figure 2: ]{\mbox{ The partition of  $(A,B)$. }}\label{fig-2}
\end{figure}

Let $A_{xz}=N(x)\cap N(z)\cap A$, $B_{xz}=N(x)\cap N(z)\cap B$, $A_{yz}=N(y)\cap N(z)\cap A$ and $B_{yz}=N(y)\cap N(z)\cap B$. Since both $G[\{x,y,z\}\cup A]$ and $G[\{x,y,z\}\cup B]$ contain odd cycles of length 5, we infer that $|A_{xz}|,|A_{yz}|, |B_{xz}|,|B_{yz}|\geq 1$. Let
 $$A_x = (N(x)\cap A)\setminus A_{xz}, A_y = (N(y)\cap A)\setminus A_{yz}, B_x = (N(x)\cap B)\setminus B_{xz}, B_y = (N(y)\cap B)\setminus B_{yz}$$
 and let
$$A_z= (N(z)\cap A)\setminus (A_{xz}\cup A_{yz}), B_z=(N(z)\cap B)\setminus (B_{xz}\cup B_{yz}).$$
Let $A'=A\setminus (A_{xz}\cup A_{yz})$ and $B'=B\setminus (B_{xz}\cup B_{yz})$ (as shown in Figure \ref{fig-2}). Clearly,  $e(\{x\}, A'\cup B')=|A_x|+|B_x|$, $e(\{y\}, A'\cup B')=|A_y|+|B_y|$ and $e(\{z\}, A'\cup B')=|A_z|+|B_z|$.
Note that   $A_x$, $A_y$, $A_z$, $B_x$, $B_y$ and $B_z$ are pairwise disjoint.

Let $D=A_{xz}\cup B_{xz}\cup A_{yz}\cup B_{yz}$. Since $G$ is triangle-free, $A_{xz}, B_{xz}, A_{yz}, B_{yz}$ are pairwise disjoint. Thus,
\begin{align}\label{eq-1}
e(D,\{x,y,z\})= 2(|A_{xz}|+|A_{yz}|+|B_{xz}|+|B_{yz}|).
\end{align}
Since $G$ is triangle-free, there is no edge between $A_{xz}$ and $B_x\cup B_z$.
Thus,
 $$e(A_{xz},B')\leq  |A_{xz}|(|B'|-|B_x|-|B_z|).$$
 Similarly,
  $e(B_{xz},A')\leq |B_{xz}|(|A'|-|A_x|-|A_z|)$. Then
 \begin{align}\label{ineq-new200}
e(A_{xz}\cup B_{xz}\cup \{x\}, A'\cup B')&=e(A_{xz},  B')+e(B_{xz}, A')+e(\{x\}, A'\cup B')\nonumber\\[5pt]
&\leq |A_{xz}|(|B'|-|B_x|-|B_z|)+|B_{xz}|(|A'|-|A_x|-|A_z|)+|A_x|+|B_x|\\[5pt]
& \leq |A_{xz}||B'|+|B_{xz}||A'|-(|A_{xz}|-1)|B_x|-(|B_{xz}|-1)|A_x|\nonumber\\[3pt]
&\qquad-|A_z|-|B_z|,\nonumber
\end{align}
where the last inequality holds because $|A_{xz}|, |B_{xz}|\geq 1$.
Similarly,
\begin{align}\label{ineq-new300}
e(A_{yz}\cup B_{yz}\cup \{y\}, A'\cup B')=&e(A_{yz},  B')+e(B_{yz}, A')+e(\{y\}, A'\cup B')\nonumber\\[5pt]
&\leq |A_{yz}|(|B'|-|B_y|-|B_z|)+|B_{yz}|(|A'|-|A_y|-|A_z|)+|A_y|+|B_y|\\[5pt]
\leq & |A_{yz}||B'|+|B_{yz}||A'|-(|A_{yz}|-1)|B_y|-(|B_{yz}|-1)|A_y|\nonumber\\[3pt]
&-|A_z|-|B_z|.\nonumber
\end{align}
Note that $e(\{z\}, A'\cup B')=|A_z|+|B_z| $. Thus we obtain that
  \begin{align}\label{eq-4}
 &e(D\cup \{x,y,z\},A'\cup B')\nonumber\\[3pt]
 =&e(A_{xz}\cup B_{xz}\cup \{x\}, A'\cup B')+e(A_{yz}\cup B_{yz}\cup \{y\}, A'\cup B')+e(\{z\}, A'\cup B')\nonumber\\[3pt]
 \leq&|A_{xz}||B'|+|B_{xz}||A'|+|A_{yz}||B'|+|B_{yz}||A'|-(|A_{xz}|-1)|B_x|-(|B_{xz}|-1)|A_x|\nonumber\\[3pt]
 &-(|A_{yz}|-1)|B_y|-(|B_{yz}|-1)|A_y|-|A_z|-|B_z| \nonumber\\[3pt]
 \leq & (|A_{xz}|+|A_{yz}|)|B'|+(|B_{xz}|+|B_{yz}|)|A'|, \end{align}
 where the last inequality holds because $|A_{xz}|,|A_{yz}|,|B_{xz}|,|B_{yz}|\geq 1$.
Since $G$ is triangle-free, $e(A_x,B_x)=0$, $e(A_y,B_y)=0$ and $e(A_z,B_z)=0$. It follows that
\begin{align}\label{ineq-100}
e(A',B')&\leq |A'||B'|-|A_x||B_x|-|A_y||B_y|-|A_z||B_z|\leq |A'||B'|.
\end{align}
Since $D\subset N(z)$ and $G$ is triangle-free, $e(G[D])=0$. Combining \eqref{eq-1}, \eqref{eq-4} and \eqref{ineq-100}, we obtain that
\begin{align*}
 e(G)&= e(D,\{x,y,z\})+1+e(D\cup \{x,y,z\},A'\cup B')+e(A',B')\\[5pt]
 &\leq  2(|A_{xz}|+|A_{yz}|+|B_{xz}|+|B_{yz}|)+1+(|A_{xz}|+|A_{yz}|)|B'|+(|B_{xz}|+|B_{yz}|)|A'|\\[3pt]
 &\qquad+|A'||B'|\\[3pt]
   &= (|A'|+|A_{xz}|+|A_{yz}|)(|B'|+|B_{xz}|+|B_{yz}|)-(|A_{xz}|+|A_{yz}|-2)(|B_{xz}|+|B_{yz}|-2)\\[5pt]
  &\qquad +5\\[3pt]
  &= |A||B|-(|A_{xz}|+|A_{yz}|-2)(|B_{xz}|+|B_{yz}|-2) +5.
\end{align*}
  Since $|A|+|B|= n-3$, we have
$|A||B|\leq \lfloor\frac{(n-3)^2}{4}\rfloor$.
Using $|A_{xz}|, |A_{yz}|, |B_{xz}|, |B_{yz}|\geq 1$, we obtain that
  \begin{align}\label{ineq-final}
 e(G)\leq \left\lfloor\frac{(n-3)^2}{4}\right\rfloor+5.
  \end{align}

Equality holds in
 \eqref{ineq-final} if and only if  equality holds  in \eqref{ineq-new200}, \eqref{ineq-new300}, \eqref{eq-4} and \eqref{ineq-100}, and $\{|A|,|B|\}= \{\lfloor\frac{n-3}{2}\rfloor,\lceil\frac{n-3}{2}\rceil\}$, $G[A',B']$ is complete bipartite.  Equality holds  in \eqref{ineq-new200}  if and only if both $G[A_{xz}, B'\setminus (B_x\cup B_z)]$ and $G[B_{xz}, A'\setminus (A_x\cup A_z)]$ are complete bipartite. Equality holds  in \eqref{ineq-new300}  if and only if  both $G[A_{yz}, B'\setminus (B_y\cup B_z)]$ and $G[B_{yz}, A'\setminus (A_y\cup A_z)]$ are complete bipartite. Equality holds  in \eqref{eq-4}  if and only if
 \begin{align}\label{ineq-200}
(|A_{xz}|-1)|B_x|=(|B_{xz}|-1)|A_x|=(|A_{yz}|-1)|B_y|=(|B_{yz}|-1)|A_y|=|A_z|=|B_z|=0.
 \end{align}
Equality holds  in \eqref{ineq-100} if and only if $|A_x||B_x|=|A_y||B_y|=0$.

Note that $N(z)$ is an independent set. Then by $\chi(G)\geq 4$ we infer that $G-N(z)$ contains an odd cycle $C'$. Then $V(C')\subset \{x,y\}\cup (A'\setminus A_z)\cup (B'\setminus  B_z)$.
Since $|A_x||B_x|=0=|A_y||B_y|$, we infer that $\{x,y\}\subset V(C')$. Since $C'$ is an odd cycle and $|A_x||B_x|=0=|A_y||B_y|$,  we infer that $xy$  is an edge on $C'$ and either $A_x\neq \emptyset \neq A_y$ or $B_x\neq \emptyset \neq B_y$.

\begin{figure}[H]
\centering
\ifpdf
  \setlength{\unitlength}{0.05 mm}%
  \begin{picture}(966.3, 710.6)(0,0)
  \put(0,0){\includegraphics{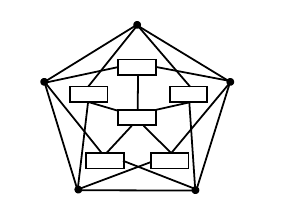}}
  \put(500.02,629.43){\fontsize{5.69}{6.83}\selectfont \makebox(20.0, 40.0)[l]{$z$\strut}}
  \put(806.26,417.59){\fontsize{5.69}{6.83}\selectfont \makebox(120.0, 40.0)[l]{$B_{xz}$\strut}}
  \put(688.96,38.80){\fontsize{5.69}{6.83}\selectfont \makebox(20.0, 40.0)[l]{$x$\strut}}
  \put(193.59,38.80){\fontsize{5.69}{6.83}\selectfont \makebox(20.0, 40.0)[l]{$y$\strut}}
  \put(40.00,418.70){\fontsize{5.69}{6.83}\selectfont \makebox(120.0, 40.0)[l]{$B_{yz}$\strut}}
  \put(634.00,542.64){\fontsize{0.69}{0.83}\selectfont \makebox(38.7, 4.8)[l]{$A'\setminus (A_x\cup A_y)$\strut}}
  \put(318.71,416.76){\fontsize{5.69}{6.83}\selectfont \makebox(120.0, 40.0)[l]{$A_{yz}$\strut}}
  \put(636.56,421.49){\fontsize{5.69}{6.83}\selectfont \makebox(120.0, 40.0)[l]{$A_{xz}$\strut}}
  \put(540.46,232.18){\fontsize{5.69}{6.83}\selectfont \makebox(60.0, 40.0)[l]{$A_y$\strut}}
  \put(322.16,224.50){\fontsize{5.69}{6.83}\selectfont \makebox(60.0, 40.0)[l]{$A_x$\strut}}
  \put(480.22,346.69){\fontsize{5.69}{6.83}\selectfont \makebox(40.0, 40.0)[l]{$B'$\strut}}
  \end{picture}%
\else
  \setlength{\unitlength}{0.05 mm}%
  \begin{picture}(966.3, 710.6)(0,0)
  \put(0,0){\includegraphics{graph2}}
  \put(500.02,629.43){\fontsize{5.69}{6.83}\selectfont \makebox(20.0, 40.0)[l]{$z$\strut}}
  \put(806.26,417.59){\fontsize{5.69}{6.83}\selectfont \makebox(120.0, 40.0)[l]{$B_{xz}$\strut}}
  \put(688.96,38.80){\fontsize{5.69}{6.83}\selectfont \makebox(20.0, 40.0)[l]{$x$\strut}}
  \put(193.59,38.80){\fontsize{5.69}{6.83}\selectfont \makebox(20.0, 40.0)[l]{$y$\strut}}
  \put(40.00,418.70){\fontsize{5.69}{6.83}\selectfont \makebox(120.0, 40.0)[l]{$B_{yz}$\strut}}
  \put(634.00,542.64){\fontsize{0.69}{0.83}\selectfont \makebox(38.7, 4.8)[l]{$A'\setminus (A_x\cup A_y)$\strut}}
  \put(318.71,416.76){\fontsize{5.69}{6.83}\selectfont \makebox(120.0, 40.0)[l]{$A_{yz}$\strut}}
  \put(636.56,421.49){\fontsize{5.69}{6.83}\selectfont \makebox(120.0, 40.0)[l]{$A_{xz}$\strut}}
  \put(540.46,232.18){\fontsize{5.69}{6.83}\selectfont \makebox(60.0, 40.0)[l]{$A_y$\strut}}
  \put(322.16,224.50){\fontsize{5.69}{6.83}\selectfont \makebox(60.0, 40.0)[l]{$A_x$\strut}}
  \put(480.22,346.69){\fontsize{5.69}{6.83}\selectfont \makebox(40.0, 40.0)[l]{$B'$\strut}}
  \end{picture}%
\fi
\caption[Figure 5]{The structure of $G$. }\label{fig-4}
\end{figure}

By symmetry assume   $A_x\neq \emptyset \neq A_y$. Then $|A_x||B_x|=|A_y||B_y|=0$ implies $|B_x|=|B_y|=0$.  By \eqref{ineq-200}, $A_x\neq \emptyset \neq A_y$ implies  $|B_{xz}|=|B_{yz}|=1$. Hence $G$ is isomorphic to a graph in  $\hg(n)$ (as shown in Figure \ref{fig-4}). Thus the theorem holds.
\end{proof}


\end{document}